\documentclass[11pt]{amsart}
\usepackage{amsmath}
\usepackage{amsfonts}
\usepackage{amsthm}
\chardef\bslash=`\\ 

\newtheorem{thm}{Theorem}[section]
\newtheorem{cor}[thm]{Corollary}

\newtheorem{prop}[thm]{Proposition}

\newtheorem{example}{Example}

\theoremstyle{definition}
\newtheorem{defn}{Definition}[section]
\newtheorem{rem}{Remark}[section]

\theoremstyle{remark}

\newcommand{\thmref}[1]{Theorem~\ref{#1}}

\newcommand{\corref}[1]{Corollary~\ref{#1}}
\newcommand{\defnref}[1]{Definition~\ref{#1}}
\newcommand{\propref}[1]{Proposition~\ref{#1}}

\newcommand{\F}{\mathcal{F}}

\newcommand{\D}{\mathbb{D}}

\newcommand{\eval}[2][\right]{\relax
  \ifx#1\right\relax \left.\fi#2#1\rvert}

\begin{document}
\title[Radical transversal lightlike hypersurfaces]{RADICAL TRANSVERSAL LIGHTLIKE HYPERSURFACES OF ALMOST COMPLEX MANIFOLDS WITH NORDEN METRIC}  
\author[Galia Nakova]{Galia Nakova}
\address{Department of Algebra and Geometry\\
Faculty of Mathematics and Informatics\\
University of Veliko Tarnovo "St. Cyril and St. Metodius"\\
2 T. Tarnovski str, \, Veliko Tarnovo 5003\\
Bulgaria}
\email{gnakova@gmail.com} 

\begin{abstract}
In this paper we introduce radical transversal lightlike hypersurfaces of almost complex manifolds with Norden metric. The
study of these hypersurfaces is motivated by the fact that for indefinite almost Hermitian manifolds this class of lightlike 
hypersurfaces does not exist. We also establish that radical transversal lightlike hypersurfaces of almost complex manifolds
with Norden metric have nice properties as a unique screen distribution and a symmetric Ricci tensor of the considered 
hypersurfaces of Kaehler manifolds with Norden metric. We obtain new results about lightlike hypersurfaces concerning to
their relations with non-degenerate hypersurfaces of almost complex manifolds with Norden metric. Examples of the
considered hypersurfaces are given.
\end{abstract}
\maketitle

\section{Introduction}\label{sec-1}

There exist two types submanifolds of a semi-Riemannian manifold $(\overline M,\overline g)$ with respect to the induced metric
$g$ by $\overline g$ on the submanifold. If $g$ is non-degenerate or degenerate, the submanifold $(M,g)$ is non-degenerate or
lightlike, respectively. In case $g$ is non-degenerate on $M$, both the tangent bundle $TM$ and the normal bundle $TM^\bot $ of
$M$ are non-degenerate and $TM\cap TM^\bot =\{0\}$. However, in case $(M,g)$ is a lightlike submanifold of $\overline M$, a part
of $TM^\bot $ lies in $TM$. Therefore the geometries of the non-degenerate and the lightlike submanifolds are different. The
general theory of lightlike submanifolds has been developed in \cite{D-B} by K. Duggal and A. Bejancu. The geometry of 
Cauchy-Riemann (CR) lightlike submanifolds of indefinite Kaehler manifolds was presented in \cite{D-B}, too. Some new classes
of lightlike submanifolds of indefinite Kaehler, Sasakian and quaternion Kaehler manifolds were introduced in \cite{D-S} by
K. Duggal and B. Sahin. In \cite{D-B}, \cite{D-S} many applications of lightlike geometry in the mathematical physics were given. 
\par
Lightlike hypersurfaces of indefinite Kaehler manifolds were studied in \cite{D-B}, \cite{D-S}.  In this paper we introduce 
radical transversal lightlike hypersurfaces of almost complex manifolds with Norden metric. Such class of lightlike
hypersurfaces does not exist when the ambient manifold is an indefinite almost Hermitian manifold because its geometry is
different from the geometry of an almost complex manifold with Norden metric. The difference arises due to the fact that the
action of the almost complex structure $\overline J$ on the tangent space at each point of an almost complex manifold with
Norden metric $\overline M$ is an anti-isometry with respect to the metric $\overline g$. The metric $\overline g$ on $\overline M$ is called Norden metric (or B-metric). Moreover, the tensor field $\overline {\widetilde g}$ on $\overline M$ defined by 
$\overline {\widetilde g}(X,Y)=\overline g(\overline JX,Y)$ is also Norden metric on $\overline M$ while in the almost Hermitian
case $\overline {\widetilde g}$ is a 2-form. Both metrics $\overline g$ and $\overline {\widetilde g}$ on $\overline M$ are of a neutral signature. The beginning of the investigations in the geometry of the almost complex manifolds with Norden metric was put
by A. P. Norden \cite{Norden} and the researches have been continued by G. Ganchev, K. Gribachev, D. Mekerov, A. Borisov, V. Mihova
(\cite{GriMekDj}, \cite{GB}, \cite{GGM}).
\par
In Section 2 we recall some preliminaries about lightlike hypersurfaces of semi-Riemannian manifolds, almost complex manifolds 
with Norden metric and almost contact manifolds with B-metric. In Section 3 we define a radical transversal lightlike
hypersurface of an almost complex manifold with Norden metric and prove that a lightlike hypersurface of such manifold is
radical transversal if and only if the screen distribution of the lightlike hypersurface is holomorphic. In Section 4 we
show that a radical transversal lightlike hypersurface of an almost complex manifold with Norden metric has a unique screen
distribution up to a semi-orthogonal transformation. This property is important for the lightlike hypersurface because it 
guarantees that the induced geometrical objects on the hypersurface do not depend on the choice of the screen distribution.
We establish that the Ricci tensor of radical transversal lightlike hypersurface of a Kaehler manifold with Norden metric 
is symmetric, which is not true in general in the lightlike geometry. We close this section by some geometrical characterizations
of the considered hypersurfaces. Since on an almost complex manifold with Norden metric there exist two Norden metrics, in 
\cite{GN} we consider submanifolds which are non-degenerate with respect to the one Norden metric and lightlike with respect to
the other one. Section 5 is devoted to the same topic. We prove that $(M,g)$ is a special non-degenerate hypersurface of an 
almost complex manifold with Norden metric $(\overline M,\overline J,\overline g,\overline {\widetilde g})$ if and only if
$(M,\widetilde g)$ is a radical transversal lightlike hypersurface of $\overline M$, where $g$ and $\widetilde g$ are the
induced metrics by $\overline g$ and $\overline {\widetilde g}$ on $M$, respectively. We find relations between the induced
geometrical objects on the hypersurfaces $(M,g)$ and $(M,\widetilde g)$ of a Kaehler manifold with Norden metric and 
characterize both hypersurfaces. In the last section we give two examples of radical transversal lightlike hypersurfaces.
  
\section{Preliminaries}\label{sec-2}
\subsection{Lightlike hypersurfaces of semi-Riemannian manifolds}
Let $M$ be a hypersurface of an $(m+2)$-dimensional semi-Riemannian manifold $(\overline M,\overline g)$ of index 
$q\in \{1,\ldots ,m+1\}$. $M$ is a lightlike hupersurface of $\overline M$ \cite{D-B} if at any $u\in M$ 
${\rm Rad} T_xM\neq \{0\}$, where ${\rm Rad} T_xM=T_xM\cap T_xM^\bot $. Because of for a hypersurface 
${\rm dim} \left(T_xM^\bot\right)=1$ it follows that ${\rm dim} \left({\rm Rad} T_xM\right)=1$ and ${\rm Rad} T_xM=T_xM^\bot $. 
${\rm Rad} TM$ is called a radical distribution on $M$. Hence, the induced metric $g$ by $\overline g$ on a lightlike hypersurface 
$M$ has a constant rank $m$. Moreover, there exists a non-degenerate complementary   
vector bundle $S(TM)$ of $TM^\bot $ in $TM$, which is called in \cite{D-B} the screen distribution on $\overline M$. For any $S(TM)$
we have a unique transversal vector bundle ${\rm tr}(TM)$ which is a lightlike complementary vector bundle (but not orthogonal) 
to $TM$ in $T\overline M$. So, the following decompositions of $T\overline M$ are valid:
\begin{equation}\label{2.1}
T\overline M=S(TM)\bot \left(TM^\bot \oplus {\rm tr}(TM)\right)=TM\oplus {\rm tr}(TM),
\end{equation}
where by $\bot $ (resp. $\oplus $) is denoted an orthogonal (resp. a non-orthogonal) direct sum. 
By $\Gamma (E)$ is denoted the ${\F}(M)$-module of smooth sections of a vector bundle $E$ over $M$, ${\F}(M)$ being the algebra of smooth functions on $M$.
In (\cite{D-B}, Theorem 1.1, p. 79) it was proved if $(M, g, S(TM))$ is a lightlike hypersurface of $\overline M$, for any non-zero section $\xi $ of $TM^\bot $ 
on a coordinate neighbourhood $U\subset M$, there exists a unique section $N$ of ${\rm tr}(TM)$ on $U$ satisfying:
\begin{equation}\label{2.2}
\overline g(N,\xi )=1, \qquad \overline g(N,N)=\overline g(N,W)=0, \, \, \forall \, W \in \Gamma (S(TM)).
\end{equation}
The induced geometrical objects on a lightlike hypersurface $M$ of a semi-Riemannian manifold $\overline M$ have different
properties from the properties of the ones  on a non-degenerate hypersurface of $\overline M$. Therefore, follow
\cite{D-B}, \cite{D-S} we will recall basic formulas and facts about the induced geometrical objects on a lightlike hypersurface. 
Let $\overline \nabla $ be the Levi-Civita connection on $\overline M$ with respect to $\overline g$.  
The global Gauss and Weingarten formulas are
\begin{equation*}
\begin{array}{ll}
\overline \nabla _XY=\nabla _XY+h(X,Y),\\
\overline \nabla _XV=-A_VX+\nabla ^t_XV, \quad \forall X, Y \in \Gamma (TM), \quad V \in \Gamma ({\rm tr}(TM)),
\end{array}
\end{equation*}
where $\nabla _XY$ and $A_VX$ belong to $\Gamma (TM)$ while $h(X,Y)$ and $\nabla ^t_XV$ belong to $\Gamma ({\rm tr}(TM))$. The induced connection $\nabla $ 
on $M$ is a torsion-free linear connection and in general $\nabla $ is not metric connection. The linear connection $\nabla ^t$ is called an induced linear 
connection on $\Gamma ({\rm tr}(TM))$. The second fundamental form $h$ is symmetric ${\F}(M)$-bilinear form on $\Gamma (TM)$. The shape operator $A_V$ is 
$\Gamma (S(TM))$-valued and it is not self-conjugate with respect to $g$, i.e. $g(A_VX,Y)\neq g(X,A_VY)$.
The local Gauss and Weingarten formulas are
\begin{equation}\label{2.3}
\begin{array}{ll}
\overline \nabla _XY=\nabla _XY+B(X,Y)N,\\
\overline \nabla _XN=-A_NX+\tau (X)N, \quad \forall \, X, Y \in \Gamma (TM_{|U}), 
\end{array}
\end{equation}
where the pair of sections $\{\xi ,N\}$ on $U\subset M$ satisfies (\ref{2.2}), $B$ is a symmetric ${\F}(U)$-bilinear form which is called the local
second fundamental form of $M$ and $\tau $ is a 1-form on $U$. We also have
\begin{equation}\label{2.4}
A_N\xi =0, \qquad B(X,\xi )=0, \, \, \forall \, X \in \Gamma (TM_{|U}).
\end{equation}
Let $P$ denote the projection morphism of $\Gamma (TM)$ on $\Gamma (S(TM))$. The following formulas are the Gauss and Weingarten
equations for the screen distribution $S(TM)$ 
\begin{equation*}
\begin{array}{ll}
\nabla _XPY=\nabla ^*_XPY+h^*(X,PY),\\
\nabla _XU=-A^*_UX+\nabla ^{*t}_XU, \quad \forall \, X, Y \in \Gamma (TM), \, \, U \in \Gamma (TM^\bot ), 
\end{array}
\end{equation*}
where $\nabla ^*_XPY$ and $A^*_UX$ belong to $\Gamma (S(TM))$, $\nabla ^*$ and $\nabla ^{*t}$ are linear connections on  
$\Gamma (S(TM))$ and $\Gamma (TM^\bot )$, respectively; $h^*$ is a $\Gamma (TM^\bot )$-valued ${\F}(M)$-bilinear form on
$\Gamma (TM)\times \Gamma (S(TM))$ and $A^*_U$ is $\Gamma (S(TM))$-valued ${\F}(M)$-bilinear operator on $\Gamma (TM)$. They
are called the screen second fundamental form and screen shape operator of $S(TM)$, respectively. Locally for any
$X, Y \in \Gamma (TM_{|U})$ we have  
\begin{equation}\label{2.5}
\nabla _XPY=\nabla ^*_XPY+C(X,PY)\xi , \quad  \nabla _X\xi =-A^*_\xi X-\tau (X)\xi ,  
\end{equation}
where $C(X,PY)$ is the local screen fundamental form of $S(TM)$. Both local second fundamental forms $B$ and $C$ are related to
their shape operators by
\begin{equation}\label{2.6}
B(X,Y)=g(A^*_\xi X,Y), \quad  C(X,PY)=g(A_NX,PY).
\end{equation}
$\nabla ^*$ is a metric connection, $A^*$ is self-conjugate with respect to $g$ and 
\begin{equation}\label{2.7}
A^*_\xi \xi =0, \qquad h^*(\xi ,PY)=0.
\end{equation}
As the screen distribution $S(TM)$ is not unique, the induced geometrical objects depend on the choice of $S(TM)$. Follow
\cite{D-B}, \cite{D-S} we will present their dependence (or otherwise) on the choice of a screen distribution.   
Let $F=\{\xi ,N,W_i\}, \, i=\{1,\ldots ,m\}$ be a quasi-orthonormal basis of $\overline M$ along $M$, where $\{\xi \}$, $\{N\}$ 
and $\{W_i\}$ are the lightlike basis of $\Gamma (\rm Rad TM_{|U})$, $\Gamma ({\rm tr}(TM)_{|U})$ and the orthonormal basis
of $\Gamma (S(TM)_{|U})$, respectively. Consider two quasi-orthonormal frames fields $F=\{\xi ,N,W_i\}$ and
$F^\prime =\{\xi ,N^\prime ,W^\prime _i\}$ induced on $U\subset M$ by $\{S(TM), {\rm tr}(TM)\}$ and $\{S^\prime (TM), 
{\rm tr}^\prime (TM)\}$, respectively for the same $\xi $. The following relationships between $F$ and $F^\prime $ are valid
\begin{equation}\label{2.8}
\begin{array}{ll}
W^\prime _i=\sum \limits_{j=1}^{m}W^j_i(W_j-\epsilon _j{\rm f}_j\xi ) , \quad
N^\prime =N-\displaystyle{\frac{1}{2}\left\{\sum \limits_{i=1}^{m}\epsilon _i({\rm f}_i)^2\right\}}\xi +
\sum \limits_{i=1}^{m}{\rm f}_iW_i \, ,
\end{array}
\end{equation}
where $\{\epsilon _1,\ldots ,\epsilon _m\}$ is the signature of the orthonormal basis $\{W_i\}$ and $W^j_i$, ${\rm f}_i$ are
smooth functions on $U$ such that $(W^j_i)$ are $m\times m$ semi-orthogonal matrices. It was proved (\cite{D-B}, \cite{D-S}) that
$B$ is independent of the choice of $S(TM)$, but both $B$ and $\tau $ depend on the choice of a section 
$\xi \in \Gamma (\rm Rad TM_{|U})$. Moreover, relationships between the induced objects  $\{\nabla , \tau , A_N, A^*_\xi , C\}$ and  
$\{\nabla ^\prime , \tau ^\prime , A^\prime _{N^\prime }, A^{*\prime }_\xi , C^\prime \}$ by the screen distributions $S(TM)$ 
and $S^\prime (TM)$, respectively, were given.
\subsection{Almost complex manifolds with Norden metric}
Let $(\overline M,\overline J,\overline g)$ be a $2n$-dimensional almost complex manifold with Norden metric 
 \cite{GB} , i.e. $\overline J$ is an almost complex structure and $\overline g$ is a metric on $\overline M$ such that:
\begin{equation*}
\overline J^2X=-X, \quad \overline g(\overline JX,\overline JY)=-\overline g(X,Y), \quad X, Y \in \Gamma (T\overline M).
\end{equation*}
The tensor field $\overline {\widetilde g}$ of type $(0,2)$ on $\overline M$ defined by 
$\overline {\widetilde g}(X,Y)=\overline g(\overline JX,Y)$ is a Norden metric on $\overline M$, too.
Both metrics $\overline g$ and $\overline {\widetilde g}$ are necessarily
of signature $(n,n)$. The metric $\overline {\widetilde g}$ is said to be an {\it associated metric} of $\overline M$.
The Levi-Civita connection of $\overline g$ is denoted by $\overline \nabla $.
The tensor field $F$ of type $(0,3)$ on $\overline M$ is defined by 
$F(X,Y,Z)=\overline g((\overline \nabla _X\overline J)Y,Z)$.
Let $\overline {\widetilde \nabla }$ be the Levi-Civita connection of $\overline {\widetilde g}$. Then
$\Phi (X,Y)=\overline {\widetilde \nabla }_XY-\overline \nabla _XY$
is a tensor field of type $(1,2)$ on $\overline M$. Since $\overline \nabla $ and $\overline {\widetilde \nabla }$ 
are torsion free we have $\Phi (X,Y)=\Phi (Y,X)$.
A classification of the almost complex manifolds with Norden metric with respect to the tensor $F$ is given in
\cite{GB} and eight classes are obtained. In \cite{GGM} these classes are characterized by conditions 
for the tensor $\Phi $. The two types of characterization conditions for the class of the Kaehler manifolds with
Norden metric are $F(X,Y,Z)=0$ and $\Phi (X,Y)=0$.

\subsection{Almost contact manifolds with B-metric}
Let $(\overline M,\varphi,\overline \xi,\overline \eta,\overline g)$ be a $(2n+1)$-dimensional almost contact
manifold with B-metric, i.~e. $(\varphi,\overline \xi,\overline \eta)$ is an almost 
contact structure \cite{Bler} and $\overline g$ is a metric \cite{GaMGri} on 
$\overline M$ such that
\begin{equation*}
\begin{array}{l}
\varphi^2X=-id+\overline \eta\otimes \overline \xi, \qquad \overline \eta(\overline \xi)=1, \\ 
\overline g(\varphi X,\varphi Y)=-\overline g(X,Y)+\overline \eta(X)\overline \eta(Y), \\
\end{array}
\end{equation*}
where $id$ denotes the identity transformation and $X, Y \in \Gamma (T\overline M)$. Immediate consequences of the above
conditions are:
\[
\overline \eta \circ \varphi =0, \quad \varphi \overline \xi =0, \quad {\rm rank}\varphi=2n, \quad \overline \eta (X)=
\overline g(X,\overline \xi ), \quad \overline g(\overline \xi ,\overline \xi )=1.
\]
The $2n$-dimensional distribution ${\D}: x\longrightarrow  {\D}_x \subset T_x\overline M$ at each point $x\in \overline M$ 
defined by ${\D}_x={\rm Ker}\eta _x$ is called a contact distribution of $\overline M$. We have the following decomposition
of $T_x\overline M$ which is orthogonal with respect to $\overline g$
\begin{equation}\label{*}
T_x\overline M={\D}_x\bot span \{\overline \xi _x\}.
\end{equation}
The tensor $\overline {\widetilde g}$ given by $\overline {\widetilde g}(X,Y)=\overline g(X,\varphi Y)+
\overline \eta (X)\overline \eta (Y)$ is a B-metric, too. Both metrics $\overline g$ and
$\overline {\widetilde g}$ are indefinite of signature $(n+1,n)$.
Let $\overline \nabla$ be the Levi-Civita connection of the metric $\overline g$. The tensor
field $F$ of type $(0,3)$ on $\overline M$ is defined by
$F(X,Y,Z)=\overline g((\overline \nabla_X\varphi)Y,Z), \quad X,Y,Z \in \Gamma (T\overline M)$.
The following 1-forms are associated with $F$
\[
\theta (X)=\overline g^{ij}F(e_i,e_j,X), \quad \theta ^*(X)=\overline g^{ij}F(e_i,\varphi e_j,X), \quad
\omega (X)=F(\overline \xi ,\overline \xi ,X),
\]
where $\{e_i,\xi \}, \, i=\{1,\ldots ,2n\}$ is a basis of $T_u\overline M$ and $(\overline g^{ij})$ is the inverse
matrix of $(\overline g_{ij})$.
A classification of the almost contact manifolds with B-metric with respect to the tensor $F$ is given in \cite{GaMGri} and 
eleven basic classes  $\F_i (i=1,2,\dots ,11)$ are obtained. 

\section{Radical transversal lightlike hypersurfaces of almost complex manifolds with Norden metric}\label{sec-3}
First in this section we will show that there are lightlike hypersurfaces of an almost complex manifold with Norden metric which 
do not exist when the ambient manifold is an indefinite almost Hermitian manifold. This fact is a motivation for our researches 
in this paper.
\par
Let $(M,g,S(TM))$ be a lightlike hypersurface of $(\overline M,\overline J,\overline g)$, where $\overline M$ is an indefinite
almost Hermitian manifold or an almost complex manifold with Norden metric. Take $\xi \in \Gamma (TM^\bot )$ and according to
(\ref{2.1}) we can write $\overline J\xi $ in the following manner
\begin{equation}\label{3.1}
\overline J\xi =\xi _1+a\xi +bN,
\end{equation}
where $\xi _1 \in \Gamma (S(TM)), \, N \in \Gamma ({\rm tr}(TM))$ and $a, b$ are smooth functions on $\overline M$. 
Since $\overline J^2=-id$, it is clear that the case $\overline J\xi =a\xi $ is impossible. From (\ref{3.1}), by using
(\ref{2.2}) we obtain $b=\overline g(\overline J\xi ,\xi )$. Now, if $\overline M$ is an indefinite almost Hermitian 
manifold, then $b=0$ and from (\ref{3.1}) it follows that $\overline J\xi $ is tangent to $M$. Thus, $\overline J(TM^\bot )$
is always a distribution on $M$ of rank 1 such that $TM^\bot \cap \overline J(TM^\bot )=\{0\}$.
\par
Further, we assume that $\overline M$ is an almost complex manifold with Norden metric. As $\overline g$
is an anti-isometry with respect to $\overline J$, the function $b$ is not zero, in general. In the case $b\neq 0$, the
component of $\overline J\xi $ with respect to $N$ does not vanish. Hence, we can consider lightlike hypersurfaces of
$\overline M$ such that $\overline J(TM^\bot )$ does not belong to $TM$. Our aim in this section is to study one class
of such lightlike hypersurfaces of $\overline M$. 
\begin{defn}\label{Definition 3.1}
Let $(M,g,S(TM))$ be a lightlike hypersurface of an almost complex manifold with Norden metric 
$(\overline M,\overline J,\overline g)$. We say that $M$ is a {\it radical transversal lightlike hypersurface} of $\overline M$
if $\overline J(TM^\bot )={\rm tr}(TM)$.
\end{defn}
\begin{rem}\label{Remark 3.1}
Radical transversal lightlike submanifolds of indefinite Kaehler manifolds were introduced by B. Sahin in \cite{Sahin}. Note 
that the dimension $r$ of the radical distribution of these submanifolds is greater than one.
\end{rem}
\begin{thm}\label{Theorem 3.1}
Let $(M,g,S(TM))$ be a lightlike hypersurface of an almost complex manifold with Norden metric 
$(\overline M,\overline J,\overline g)$. $M$ is a radical transversal lightlike hypersurface of $\overline M$ iff
the screen distribution $S(TM)$ is holomorphic with respect to $\overline J$.
\end{thm}
\begin{proof}
Let $M$ be a radical transversal lightlike hypersurface of $\overline M$. Then $\overline J\xi =bN$, where the pair
$\{\xi \in \Gamma(TM^\bot ), \, \, N \in \Gamma ({\rm tr}(TM))\}$ satisfies (\ref{2.2}) and $b \in {\F}(\overline M)$. Hence, for an arbitrary $W \in \Gamma (S(TM))$ we have $\overline g(\overline JW,\xi )=\overline g(W,bN)=0$. Thus $\overline JW$ is tangent to $M$.
Moreover, we have $\overline g(\overline JW,N)=\displaystyle{\overline g\left(W,-\frac{1}{b}\xi \right)=0}$, which implies 
$\overline JW \in \Gamma (S(TM))$, i.e. $S(TM)$ is holomorphic.
 Conversely, let $S(TM)$ be holomorphic. Taking into account 
$\overline g(\overline JW,\xi )=0$ and the decomposition (\ref{3.1}), we obtain $\overline g(W,\xi _1)=0$. Since $\overline g$
is non-degenerate on $S(TM)$, from the last equality it follows that $\xi _1=0$. Then (\ref{3.1}) becomes
\begin{equation}\label{3.2}
\overline J\xi =a\xi +bN .
\end{equation}
As $\overline g(\overline J\xi ,\overline J\xi )=0$, by using (\ref{3.2}) we have $2ab=0$. The function $b$ in (\ref{3.2})
is not zero because $TM^\bot \cap \overline J(TM^\bot )=\{0\}$. Therefore $a=0$ and $\overline J\xi =bN$ which means that
$M$ is a radical transversal lightlike hypersurface of $\overline M$.
\end{proof}
In (\cite{D-B}, p. 194) it was proved that a lightlike hypersurface of an indefinite Hermitian manifold is a CR-manifold.
In order to prove our next result, analogously as in \cite{D-B}, we will use the following 
\begin{thm}{\rm (\cite{D-B}, p. 193)}\label{Theorem 3.2}
A smooth manifold $L$ is a CR-manifold if and only if, it is endowed with an almost complex distribution
$(D,J)$ (i.e. $J(D)=D$) such that 
\begin{equation}\label{3.3}
[JX,JY]-[X,Y] \in D
\end{equation}
and
\begin{equation}\label{3.4}
N_J(X,Y)=0,
\end{equation}
for all $X, Y \in D$.
\end{thm}
\begin{thm}\label{Theorem 3.3}
A radical transversal lightlike hypersurface $(M,g,S(TM))$ of a complex manifold with Norden metric 
$(\overline M,\overline J,\overline g)$ is a CR-manifold.
\end{thm}
\begin{proof}
From \thmref{Theorem 3.1} it follows that $S(TM)$ is an almost complex distribution on $M$ with an almost complex structure
$J$ which is the restriction of $\overline J$ on $S(TM)$. Further, we will show that $S(TM)$ satisfies the conditions
(\ref{3.3}) and (\ref{3.4}). Denote by $\overline N_{\overline J}$ and $N_J$ the Nijenhuis tensors of $\overline J$ and $J$,
respectively. As $\overline M$ is a complex manifold, $\overline N_{\overline J}=0$ on $\overline M$, i.e.
\[
\overline N_{\overline J}(X,Y)=\left[\overline JX,\overline JY\right]-\left[X,Y\right]-\overline J\left(\left[X,\overline JY\right]
+\left[\overline JX,Y\right]\right)=0,
\] 
$\forall X,Y\in \Gamma (T\overline M)$. Then for any $X, Y \in \Gamma (S(TM))$ we have
\begin{equation}\label{3.5}
\begin{array}{ll}
\overline N_{\overline J}(X,Y)=[JX,JY]-[X,Y]-J(P([X,JY]+[JX,Y])) \\
\qquad \qquad \quad-\overline J(Q([X,JY]+[JX,Y]))=0,
\end{array}
\end{equation}
where $P$ and $Q$ are the projection morphisms of $TM$ on $S(TM)$ and $TM^\bot $, respectively. 
Because of the second fundamental form $h$ is symmetric on $TM$, from the Gauss formula we obtain
$[X,Y] \in \Gamma (TM)$ for any $X, Y \in \Gamma (TM)$. This fact and $J(S(TM))=S(TM)$ imply that 
for any $X, Y \in \Gamma (S(TM))$, the vector field $Z=[JX,JY]-[X,Y]-J(P([X,JY]+[JX,Y]))$ is tangent
to $M$. Since $M$ is a radical transversal lightlike hypersurface of $\overline M$, we have that the
vector field $V=\overline J(Q([X,JY]+[JX,Y]))$ belongs to $\Gamma ({\rm tr}(TM))$. From the equality
(\ref{3.5}) it follows that the components $Z$ and $V$ of $\overline N_{\overline J}$ with respect to $TM$ and 
${\rm tr}(TM)$, respectively, are both zero. The vanishing of $V$ shows that $Q([X,JY]+[JX,Y])=0$ and therefore
$[X,JY]+[JX,Y]=P([X,JY]+[JX,Y])$. Hence, $Z$ becomes 
\begin{equation}\label{3.6}
Z=[JX,JY]-[X,Y]-J([X,JY]+[JX,Y])=0,
\end{equation}
for any $X, Y \in \Gamma (S(TM))$. The condition (\ref{3.4}) is valid because the expression for $Z$ from 
(\ref{3.6}) is exactly $N_J$ on $S(TM)$. Moreover, by using of (\ref{3.6}) we obtain $[JX,JY]-[X,Y]=J([X,JY]+[JX,Y])$,
i.e. the condition (\ref{3.3}) is true for any $X, Y \in \Gamma (S(TM))$. Then our assertion follows from \thmref{Theorem 3.2}.
 \end{proof}

\section{The induced geometrical objects on a radical transversal lightlike hypersurface of a Kaehler manifold 
with Norden metric}\label{sec-4}
It is known that the induced geometrical objects on a lightlike hypersurface $M$ are well-defined if $M$ admits a unique
or canonical screen distribution. Now we will investigate this important problem for the introduced lightlike hypersurfaces 
in the previous section. We state
\begin{thm}\label{Theorem 4.1} A radical transversal lightlike hypersurface $(M,g,S(TM))$
of a $2n$-dimensional almost complex manifold with Norden metric $(\overline M,\overline J,\overline g)$ 
has a unique screen distribution up to a semi-orthogonal transformation and a unique lightlike transversal vector bundle.
\end{thm}
\begin{proof}
Let $S(TM)$ and $S^\prime (TM)$ be two screen distributions on $M$, ${\rm tr}(TM)$ and ${\rm tr}^\prime (TM)$ their lightlike
transversal vector bundles, respectively. Take the quasi-orthonormal frames fields $F=\{\xi ,N,W_i\}$ and
$F^\prime =\{\xi ,N^\prime ,W^\prime _i\}$ induced on $U\subset M$ by $\{S(TM), {\rm tr}(TM)\}$ and $\{S^\prime (TM), 
{\rm tr}^\prime (TM)\}$, respectively. According to \thmref{Theorem 3.1} $S(TM)$ and $S^\prime (TM)$ are holomorphic. By using this 
fact and $\displaystyle{N=\frac{1}{b}\overline J\xi }$ we compute $\overline g(W^\prime _i,N)=0$ for any $i\in \{1,\ldots ,2n-2\}$.
Thus, after multiplication by $N$ both sides of the first equality in (\ref{2.8}) we get 
$\sum \limits_{j=1}^{2n-2}W^j_i\epsilon _j{\rm f}_j=0\quad (i\in \{1,\ldots ,2n-2\})$, where $(W^j_i)$ is an matrix of
$S(T_xM)$ at any point $x$ of $M$, belonging to $O(n-1, n-1)$. The determinant of the last homogeneous linear system does 
not vanish at any $x\in M$ and hence it has the unique solution ${\rm f}_j=0, \, j\in \{1,\ldots ,2n-2\}$. Then (\ref{2.8})
become $W^\prime _i=\sum \limits_{j=1}^{2n-2}W^j_iW_j\quad (i\in \{1,\ldots ,2n-2\})$, \, $N^\prime =N$, which proves
our assertion. 
\end{proof}
Let $(M,g)$ be a radical transversal lightlike hypersurface of a Kaehler manifold with Norden metric 
$(\overline M,\overline J,\overline g)$. According to \defnref{Definition 3.1} we have $\overline J\xi =bN$. By using
(\ref{2.3}) and the second equality in (\ref{2.5}), for any $X\in \Gamma (TM)$ we compute
\begin{equation}\label{4.1}
\left(\overline \nabla _X\overline J\right)N=\displaystyle{\frac{1}{b}A^*_\xi X+\overline J(A_NX)+
\frac{1}{b}\left(2\tau (X)+\frac{1}{b}(X\circ b)\right)\xi .}
\end{equation}
As $\overline M$ is a Kaehler manifold with Norden metric, the left side of (\ref{4.1}) vanishes. From $A_NX\in \Gamma (S(TM))$
and \thmref{Theorem 3.1} we have $\overline J(A_NX)\in \Gamma (S(TM))$. Then (\ref{4.1}) implies
\begin{equation}\label{4.2}
A^*_\xi X=-b\overline J(A_NX),
\end{equation}
\begin{equation}\label{4.3}
\tau (X)=-\displaystyle{\frac{1}{2b}(X\circ b).}
\end{equation}
An arbitrary $Y\in \Gamma (TM)$ can be decomposed in the following manner
\[
Y=PY+QY=PY+\eta (Y)\xi ,
\]
where $\eta $ is a 1-form on $M$ and $\eta (Y)=\overline g(Y,N)$. Hence for $\overline JY$ we have
\begin{equation}\label{4.4}
\overline JY=\overline J(PY)+\eta (Y)bN.
\end{equation}
By using (\ref{2.3}), (\ref{2.5}), (\ref{4.3}) and (\ref{4.4}) we obtain
\begin{equation}\label{4.5}
\begin{array}{lll}
\left(\overline \nabla _X\overline J\right)Y=\nabla ^*_X\overline J(PY)-b\eta (Y)A_NX-\overline J
\left(P\left(\nabla _XY\right)\right) \\
\qquad \qquad \, \, \, +\displaystyle{\left(C(X,\overline J(PY))+\frac{1}{b}B(X,Y)\right)\xi } \\
\qquad \quad \quad \, \, \, +\displaystyle{\left(B(X,\overline J(PY))+\frac{1}{2}\eta (Y)(X\circ b)+b(\nabla _X\eta )Y\right)N}.
\end{array}
\end{equation}
Since $\left(\overline \nabla _X\overline J\right)Y=0$, the parts belonging to $S(TM)$, $TM^\bot $ and ${\rm tr}(TM)$ of the
right side of (\ref{4.5}) vanish and we have
\begin{equation}\label{4.6}
\nabla ^*_X\overline J(PY)=b\eta (Y)A_NX+\overline J\left(P\left(\nabla _XY\right)\right),
\end{equation}
\begin{equation}\label{4.7}
C(X,\overline J(PY))=-\frac{1}{b}B(X,Y),
\end{equation}
\begin{equation}\label{4.8}
B(X,\overline J(PY))=-\frac{1}{2}\eta (Y)(X\circ b)-b(\nabla _X\eta )Y.
\end{equation}
Substituting $\overline J(PY)$ for $Y$ in (\ref{4.6}), (\ref{4.7}) and taking into account that $P(\overline J(PY))=
\overline J(PY)$, $\eta (\overline J(PY))=0$ and (\ref{4.8}) we find 
\begin{equation}\label{4.9}
\nabla ^*_XPY=-\overline J(P(\nabla _X\overline J(PY))),
\end{equation}
\begin{equation}\label{4.10}
C(X,PY)=\displaystyle{-\frac{1}{2b}\eta (Y)(X\circ b)-(\nabla _X\eta )Y}.
\end{equation}
Having in mind (\ref{4.9}), (\ref{4.10}), (\ref{4.2}) and (\ref{4.3}), the formulas (\ref{2.5}) become
\begin{equation}\label{4.11}
\begin{array}{ll}
\nabla _XPY=-\overline J(P(\nabla _X\overline J(PY)))-\displaystyle{\left(\frac{1}{2b}\eta (Y)(X\circ b)+
(\nabla _X\eta )Y\right)}\xi , \\
\nabla _X\xi =b\overline J(A_NX)+\displaystyle{\frac{1}{2b}(X\circ b)\xi }.
\end{array}
\end{equation}
From (\ref{4.9}) by direct calculations we obtain 
\begin{equation}\label{4.12}
(\nabla ^*_X\overline J)PY=0
\end{equation}
for any $X, Y\in \Gamma (TM)$.
\begin{thm}\label{Theorem 4.1}
Let $(M,g,S(TM))$ be a radical transversal lightlike hypersurface of a Kaehler manifold with Norden metric 
$(\overline M,\overline J,\overline g)$. The shape operator $A_N$ is self-conjugate with respect to $g$ iff
$A_N$ commutes with the action of the almost complex structure $\overline J$ on $S(TM)$.
\end{thm} 
\begin{proof}
As $A^*$ is self-conjugate with respect to $g$, by using (\ref{4.2}) we have
\begin{equation}\label{4.13}
g(\overline J(A_NX),Y)=g(X,\overline J(A_NY))
\end{equation}
for any $X, Y\in \Gamma (TM)$.
Let $A_N$ be self-conjugate with respect to $g$ on $S(TM)$. Then for any $X, Y\in \Gamma (S(TM))$ we obtain
\[
g(\overline J(A_NX),Y)=g(A_NX,\overline JY)=g(X,A_N\overline JY).
\]
The last equality and (\ref{4.13}) imply $g(X,A_N\overline JY)=g(X,\overline J(A_NY))$. As $g$ is non-degenerate on
$S(TM)$ it follows that $A_N\circ \overline J=\overline J\circ A_N$. Conversely, if $A_N\circ \overline J=\overline J\circ A_N$
on $S(TM)$, by using (\ref{4.13}) we compute
\begin{equation}\label{4.14}
g(A_N\overline JX,Y)=g(X,\overline J(A_NY)), \quad X, Y \in \Gamma (S(TM)).
\end{equation}
Replacing $X$ from (\ref{4.14}) by $\overline JX$ we obtain $g(A_NX,Y)=g(X,A_NY)$, i.e. $A_N$ is self-conjugate with respect to $g$.
\end{proof}
Further, using well known results for lightlike hypersurfaces from \cite{D-B}, \cite{D-S} and the ones obtained in this section, 
we will give some geometrical characterizations of the considered hypersurfaces.
\par
An immediate consequence from (\cite{D-B}, Theorem 2.3, p. 89) and \thmref{Theorem 4.1} is the following
\begin{cor}\label{Corollary 4.1} Let $(M,g,S(TM))$ be a radical transversal lightlike hypersurface of a Kaehler
manifold with Norden metric $(\overline M,\overline J,\overline g)$. Then the following assertions are equivalent:
\begin{enumerate}
\renewcommand{\labelenumi}{(\roman{enumi})}
\item $S(TM)$ is an integrable distribution.
\item $h^*(X,Y)=h^*(Y,X), \quad \forall X, Y \in \Gamma (S(TM))$.
\item $A_V\circ (\overline JX)=\overline J\circ (A_VX), \quad \forall X\in \Gamma (S(TM)), \quad \forall V\in
\Gamma ({\rm tr}(TM))$.
\end{enumerate}
\end{cor}
The Ricci tensor of the lightlike hypersurface $(M,g,S(TM))$ was defined in (\cite{D-B}, p. 95) by
$Ric (X,Y)={\rm trace}\{Z\longrightarrow R(X,Z)Y\}, \forall X, Y \in \Gamma (TM)$, where $R$ is the curvature tensor of $\nabla $.
In general, Ricci tensor of $M$ is not symmetric because the induced connection $\nabla $ is not a metric connection. According
to (\cite{D-B}, Theorem 3.2, p. 99) a necessary and sufficient condition the Ricci tensor of the induced connection $\nabla $
to be symmetric is each 1-form $\tau $ induced by $S(TM)$ to be closed, i.e. $d\tau =0$ on $M$. Now, taking into account
(\ref{4.3}) we state
\begin{prop}\label{Proposition 4.1}
Tne Ricci tensor of the induced connection $\nabla $ on a radical transversal lightlike hypersurface $(M,g,S(TM))$ of a
Kaehler manifold with Norden metric $(\overline M,\overline J,\overline g)$ is symmetric.
\end{prop}
\begin{thm}\label{Theorem 4.2}
Let $(M,g,S(TM))$ be a radical transversal lightlike hypersurface of a Kaehler manifold with Norden metric 
$(\overline M,\overline J,\overline g)$. Then the following assertions are equivalent:
\begin{enumerate}
\renewcommand{\labelenumi}{(\roman{enumi})}
\item $M$ is totally geodesic.
\item $S(TM)$ is totally geodesic.
\item $(\nabla _X\eta )Y=\eta (Y)\tau (X)=-\displaystyle{\frac{1}{2b}\eta (Y)(X\circ b)}, \quad \forall X, Y 
\in \Gamma (TM)$.
\end{enumerate}
\end{thm} 
\begin{proof}
According to (\cite{D-B}, Theorem 2.2, p. 88) we have that $M$ is totally geodesic if and only if $h$ vanishes identically on $M$.
From $h(X,Y)=B(X,Y)N, \, X, Y\in \Gamma (TM)$ it follows that $(i)$ is equivalent to $B(X,Y)=0$. 
The screen distribution $S(TM)$ is totally geodesic (\cite{D-B}, p. 110) if and only if $C(X,PY)=0, \, 
\forall X, Y\in \Gamma (TM)$. Then the equivalence of ({\rm i}) and ({\rm ii}) we obtain from (\ref{4.7}). The equality 
(\ref{4.10}) implies the equivalence of ({\rm ii}) and ({\rm iii}).
\end{proof}
\begin{thm}\label{Theorem 4.3}
Let $(M,g,S(TM))$ be a radical transversal lightlike hypersurface of a Kaehler manifold with Norden metric 
$(\overline M,\overline J,\overline g)$. Then
\begin{enumerate}
\renewcommand{\labelenumi}{(\roman{enumi})}
\item $M$ is totally umbilical iff $A_N(PX)=\displaystyle{\frac{\rho }{b}\overline J(PX)}, \, \forall X\in \Gamma (TM)$ and 
$\rho \in {\F}(M)$.
\item $S(TM)$ is totally umbilical iff $A^*_\xi (PX)=-bk\overline J(PX), \, \forall X\in \Gamma (TM)$ and $k\in {\F}(M)$.
\end{enumerate}
\end{thm}
\begin{proof}
As $A^*_\xi \xi =A_N\xi =0$, the equality (\ref{4.2}) is equivalent to 
\begin{equation}\label{4.15}
A^*_\xi (PX)=-b\overline J(A_N(PX)), \, \forall X\in \Gamma (TM).
\end{equation}
According to (\cite{D-B}, p. 107, p. 110), $M$ and $S(TM)$ are totally umbilical if and only if, on each $U$ of $M$ there
exists a smooth functions $\rho $ and $k$ such that 
\begin{equation}\label{4.16}
A^*_\xi (PX)=\rho PX
\end{equation}
and
\begin{equation}\label{4.17}
A_NX=kPX,
\end{equation}
for any $X\in \Gamma (TM)$, respectively. By using (\ref{4.15}), (\ref{4.16}) and (\ref{4.15}), (\ref{4.17}) we establish
the truth of the assertions ({\rm i}) and ({\rm ii}), respectively.
\end{proof}

\section{Hypersurfaces of an almost complex manifold with Norden metric which are non-degenerate with respect to the one Norden
metric and lightlike with respect to the other one}\label{sec-5}
Let $(\overline M,\overline J,\overline g,\overline {\widetilde g})$ be a $2n$-dimensional almost complex manifold with Norden 
metric and $M$ be a $(2n-1)$-dimensional hypersurface of $\overline M$. An essential difference between an indefinite almost
Hermitian manifold and an almost complex manifold with Norden metric is that there exist two Norden metrics $\overline g$ and
$\overline {\widetilde g}$ on the manifold of the second type. Hence, we can consider two induced metrics $g$ and $\widetilde g$ 
on $M$ by $\overline g$ and $\overline {\widetilde g}$, respectively. In \cite{GN} we have studied submanifolds of an almost
complex manifold with Norden metric which are non-degenerate with respect to the one Norden metric and lightlike with respect to the other one. Our aim in this section is to show how the hypersurfaces $(M,g)$ and $(M,\widetilde g)$ of $\overline M$ are related. We
note that $TM=\bigcup \limits_{x \in M} T_xM$ is the tangent bundle of both $(M,g)$ and $(M,\widetilde g)$. We will denote: the normal bundle of $(M,g)$ and $(M,\widetilde g)$ by $TM^\bot $ and $TM^{\widetilde \bot }$, respectively; an orthogonal direct sum with 
respect to $\overline g$ (resp. $\overline {\widetilde g}$) by $\bot $ (resp. $\widetilde \bot $) and a non-orthogonal direct 
sum by $\oplus $ (resp. $\widetilde \oplus $). 
\par
We consider a non-degenerate hypersurface $(M,g)$ of $\overline M$ defined by the following conditions
\begin{equation}\label{4.18}
\overline g(\overline N,\overline N)=\epsilon , \, \, \epsilon =\pm 1;\quad \overline g(\overline N,\overline J\overline N)=0,
\end{equation}
where $\overline N$ is the normal vector field to $M$. In the case when $\overline N$ is a time-like unit to $M$\,
$(\epsilon =-1)$, the hypersurface $(M,g)$ was called in (\cite{GaMGri}, \cite{MM}) {\it an isotropic hypersurface regarding
the associated metric} $\overline {\widetilde g}$ of $\overline M$. 
\begin{thm}\label{Theorem 5.1}
Let $(\overline M,\overline J,\overline g,\overline {\widetilde g})$ be an almost complex manifold with Norden 
metric and $M$ be a hypersurface of $\overline M$. $(M,g)$ is a non-degenerate hypersurface defined by (\ref{4.18})
iff $(M,\widetilde g)$ is a radical transversal lightlike hypersurface.
\end{thm}
\begin{proof}
Let $(M,g)$ be a non-degenerate hypersurface of $\overline M$ with normal vector field $\overline N$ satisfying (\ref{4.18}).
From $g(\overline N,\overline J\overline N)=0$ it follows $\overline J\overline N\in \Gamma (TM)$, i.e. $\overline J$ 
transforms the normal bundle $TM^\bot $ of $(M,g)$ in $TM$ and $\{\overline J\overline N\}$ is a basis of 
$\overline J(TM^\bot )$. Take $X\in \Gamma (TM)$ and $V=\lambda \overline J\overline N\in \Gamma (\overline J(TM^\bot ))$, \,  
$\lambda \in {\F}(M)$ we compute $\overline {\widetilde g}(X,V)=\lambda \overline g(\overline JX,\overline J\overline N)=
-\lambda \overline g(X,\overline N)=0$. The last equality implies $V$ belongs to the normal bundle $TM^{\widetilde \bot }$ of
$(M,\widetilde g)$ and consequently $\overline J(TM^\bot )\subseteq TM^{\widetilde \bot }$. Now, if 
$U\in \Gamma (TM^{\widetilde \bot })$ we have $\overline {\widetilde g}(X,U)=0$ for any $X\in \Gamma (TM)$, which is equivalent
to $\overline g(X,\overline JU)=0, \, \forall X\in \Gamma (TM)$. Hence, $\overline JU\in \Gamma (TM^\bot )$ which implies 
$TM^{\widetilde \bot }\subseteq \overline J(TM^\bot )$. So, we obtain that $TM^{\widetilde \bot }=\overline J(TM^\bot )$. As
$\overline J(TM)$ is an 1-dimensional subbundle of $TM$ it follows $TM\cap TM^{\widetilde \bot }=TM^{\widetilde \bot }=
{\rm Rad} TM$, i.e. $(M,\widetilde g)$ is a lightlike hypersurface of $\overline M$. Because of $\overline J(TM^\bot )$ is
a non-degenerate subbundle of $TM$ with respect to $\overline g$, we put $TM=\overline J(TM^\bot )\bot D$, where $D$ is
the complementary orthogonal with respect to $\overline g$ vector subbundle of $\overline J(TM^\bot )$ in $TM$. Take 
$X\in \Gamma (D)$ we compute $\overline g(\overline JX,\overline N)=\overline g(X,\overline J\overline N)=0$ which means that 
$\overline JX\in \Gamma (TM)$. Moreover, $\overline g(\overline JX,\overline J\overline N)=0$ and consequently 
$\overline JX\in \Gamma (D)$. So, we establish that $D$ is holomorphic by the action of $\overline J$. Then from (\cite{GN}, 
Lemma 3.1) it follows that both metrics $g$ and $\widetilde g$ are non-degenerate on $D$. We also have 
$\overline {\widetilde g}(X,\overline J\overline N)=\overline g(\overline J\overline X,\overline J\overline N)=0$ for any
$X\in \Gamma (D)$ which means $D\widetilde \bot TM^{\widetilde \bot}$. Thus, we conclude the vector bundle $D$ is a screen
distribution of $(M,\widetilde g)$. As $D$ is holomorphic from \thmref{Theorem 3.1} it follows that $(M,\widetilde g)$ is a
radical transversal lightlike hypersurface of $\overline M$ and ${\rm tr}(TM)=TM^\bot $. We note that for any 
$\xi \in \Gamma (TM^{\widetilde \bot})$ and $N\in \Gamma ({\rm tr}(TM))$ we have $\xi =\lambda \overline J\overline N$ and
$N=\mu \overline N$, where $\lambda , \mu \in {\F}(\overline M)$. The pair $\{\xi ,N\}$ on $(M,\widetilde g)$ satisfies
the conditions $\overline {\widetilde g}(\xi ,\xi)=\overline {\widetilde g}(N,N)=\overline {\widetilde g}(W,N)=0, \, 
\forall W\in \Gamma (S(TM))$. In the case $\overline N$ is space-like (resp. time-like), the condition 
$\overline {\widetilde g}(\xi ,N)=1$ is fulfilled by $\lambda \mu =-1$  (resp. $\lambda \mu =1 $). Conversely, let 
$(M,\widetilde g,S(TM))$ be a radical transversal lightlike hypersurface of $\overline M$. Hence, we have 
$\overline J(TM^{\widetilde \bot})={\rm tr}(TM)$ and according to \thmref{Theorem 3.1} $\overline JS(TM)=S(TM)$. For any
$X\in \Gamma (S(TM))$, $\xi \in \Gamma (TM^{\widetilde \bot})$ and $N\in \Gamma ({\rm tr}(TM))$ we get
\begin{equation*}
\begin{array}{ll}
\overline g(X,N)=-\overline {\widetilde g}(\overline JX,N)=0, \qquad \overline g(\xi ,N)=
-\overline {\widetilde g}(\overline J\xi ,N)=0, \\
 \overline g(X,\xi )=-\overline {\widetilde g}(X,\overline J\xi )=0.
\end{array}
\end{equation*}
The above three equalities imply the vector bundles $S(TM)$, $TM^{\widetilde \bot}$ and ${\rm tr}(TM)$ are mutually orthogonal
with respect to $\overline g$. Then the following decomposition of $T\overline M$ is valid
\begin{equation}\label{4.19}
T\overline M=S(TM)\bot TM^{\widetilde \bot}\bot {\rm tr}(TM)=TM\bot {\rm tr}(TM).
\end{equation}
From (\ref{4.19}) it follows that the normal bundle $TM^\bot $ of the hypersurface $(M,g)$ coincides with the transversal vector
bundle ${\rm tr}(TM)$ of $(M,\widetilde g)$ and both $TM$ and $TM^\bot $ are non-degenerate with respect to $\overline g$ which
means that $(M,g)$ is a non-degenerate hypersurface of $\overline M$. Now, let $\{\xi ,N\}$ be a pair of sections
on $(M,\widetilde g)$
satisfying the conditions (\ref{2.2}) and $\overline J\xi =bN$, \, $b\in {\F}(\overline M)$. In the case $b>0$ (resp. $b<0$)
the vector field $\overline N=\pm \sqrt{b}N$ (resp. $\overline N=\pm \sqrt{-b}N$) is a space-like (resp. time-like) normal
unit to $(M,g)$ and $\overline g(\overline N,\overline J\overline N)=0$ in both cases, which completes the proof.
\end{proof}
An isotropic hypersurface $(M,g)$ regarding the associated metric $\overline {\widetilde g}$ of an almost complex manifold with
Norden metric $(\overline M,\overline J,\overline g,\overline {\widetilde g})$, equipped with the almost contact 
$B$-metric structure
\begin{equation}\label{4.20}
\varphi :=\overline J+\overline g(. , \overline J\overline N)\overline N, \quad \overline \xi :=-\overline J\overline N, \quad
\overline \eta :=-\overline g(. , \overline J\overline N), \quad g:=\overline g_{|M}
\end{equation}
was called in \cite{MM} {\it a hypersurface of second type} of $\overline M$. In \cite{MM} it was proved that every 
hypersurface of second type $(M,g)$ of a Kaehler manifold with Norden metric $\overline M$ is an almost contact manifold
with $B$-metric belonging to the class ${\F}_4\oplus {\F}_5\oplus {\F}_6\oplus {\F}_8$. Some classes of these hypersurfaces
were characterized by the shape operator $A$. Below we recall the characterization conditions of the following classes
\begin{equation}\label{4.21}
\begin{array}{lll}
{\F}_0 : A=0; \qquad \qquad {\F}_4 : A=\displaystyle{-\frac{\theta (\overline \xi )}{2n-2}\varphi ^2, \, \, \,
\theta (\overline \xi )
={\rm tr}A}; \\
{\F}_5 : A=\displaystyle{-\frac{\theta ^*(\overline \xi )}{2n-2}\varphi , \quad \theta ^*(\overline \xi )
={\rm tr}(A\circ \varphi )}; \\ \\
{\F}_4\oplus {\F}_5\oplus {\F}_6 : A\overline \xi =0, \quad A\circ \varphi =\varphi \circ A.
\end{array}
\end{equation}
From now on in this section $(\overline M,\overline J,\overline g,\overline {\widetilde g})$ will stand for a $2n$-dimensional
Kaehler manifold with Norden metric, $(M,\varphi ,\overline \xi ,\overline \eta ,g)$ and $(M,\widetilde g)$ - a hypersurface
of second type of $\overline M$ and its corresponding radical transversal lightlike hypersurface from \thmref{Theorem 5.1}, respectively. Follow the proof of \thmref{Theorem 5.1} and taking into account the definitions of $\overline \xi $ and
$\overline \eta $ given in (\ref{4.20}) we have ${\rm tr}(TM)=TM^\bot =span \{\overline N\}$, \, 
$TM^{\widetilde \bot }=\overline J(TM^\bot )=span \{\overline \xi \}$ and
the contact distribution ${\D}$ of $(M,\varphi ,\overline \xi ,\overline \eta ,g)$ coincides with the unique 
(according to \thmref{Theorem 4.1}) screen distribution $S(TM)$ of $(M,\widetilde g)$. By using (\ref{*}) and (\ref{2.1}) we
conclude that the tangent bundles $TM$ and $T\overline M$ can be decomposed in direct sums as follows
\begin{equation}\label{5.51}
TM={\D}\bot span \{\overline \xi \}, \qquad TM={\D}\widetilde \bot span \{\overline \xi \} 
\end{equation}
and
\begin{equation}\label{4.22}
T\overline M={\D}\bot span \{\overline \xi \}\bot span \{\overline N\} , \qquad
T\overline M={\D}\widetilde \bot \left(span \{\overline \xi \}\widetilde \oplus span \{\overline N\} \right).
\end{equation}
We note that any $X\in \Gamma (TM)$ can be written as $X=PX+\overline \eta (X)\overline \xi $, 
where $PX\in {\D}$. Hence, $\varphi X=\varphi (PX)$. On the other hand, by using (\ref{4.20}) we compute $\varphi (PX)=
\overline J(PX)$, i.e. we have
\begin{equation}\label{52}
\varphi X=\varphi (PX)=\overline J(PX), \quad \forall \, X\in \Gamma (TM).
\end{equation}
Further, we will find relations between the induced geometrical objects on the hypersurfaces $(M,\varphi ,\overline \xi ,\overline \eta ,g)$ and $(M,\widetilde g)$ of $\overline M$. Let $\overline \nabla $, $\overline {\widetilde \nabla }$ be the Levi-Civita
connections of the metrics $\overline g$, $\overline {\widetilde g}$ on $\overline M$, respectively, and $\nabla $, 
$\widetilde \nabla $ be the induced linear connections on $(M,\varphi ,\overline \xi ,\overline \eta ,g)$, $(M,\widetilde g)$,
respectively. According to \cite{MM}, the formulas of Gauss and Weingarten for $(M,\varphi ,\overline \xi ,\overline \eta ,g)$
are
\begin{equation}\label{4.23}
\begin{array}{ll}
\overline \nabla _XY=\nabla _XY-g(A_{\overline N}X,Y)\overline N, \\ 
\overline \nabla _X\overline N=-A_{\overline N}X, \qquad \qquad \qquad \forall \, X,Y\in \Gamma (TM),
\end{array}
\end{equation}
where the shape operator $A_{\overline N}$ satisfies
\begin{equation}\label{4.24}
\overline \eta (A_{\overline N}X)=0 \Longleftrightarrow A_{\overline N}\overline \xi =0.
\end{equation}
As we have shown in the proof of \thmref{Theorem 5.1}, the pair of sections $\{\xi ,N\}$ on $(M,\widetilde g)$ defined by
\begin{equation}\label{4.25}
\xi =\frac{1}{\lambda }\overline J\overline N=-\frac{1}{\lambda }\overline \xi , \quad N=\lambda \overline N, \quad 
\lambda \in {\F}(\overline M)
\end{equation}
satisfies the conditions (\ref{2.2}). By using (\ref{4.25}) the formulas of Gauss and Weingarten (\ref{2.3}) for 
$(M,\widetilde g)$ become
\begin{equation}\label{4.26}
\begin{array}{ll}
\overline {\widetilde \nabla }_XY=\widetilde \nabla _XY+\lambda B(X,Y)\overline N, \\
\overline {\widetilde \nabla }_X\overline N=-{\widetilde A}_{\overline N}X+\displaystyle{\left(\tau (X)-
\frac{1}{\lambda }(X\circ \lambda )\right)\overline N}, \qquad \forall \, X,Y\in \Gamma (TM),
\end{array}
\end{equation}
where $\widetilde A$ is the shape operator of $(M,\widetilde g)$.
Taking into account that on $\overline M$ the Levi-Civita connections $\overline \nabla $ and $\overline {\widetilde \nabla }$ coincide, the formulas (\ref{4.23}), (\ref{4.26}) and the decompositions (\ref{4.22}) we get
\begin{equation}\label{4.27}
\begin{array}{ll}
{\widetilde \nabla }_XY=\nabla _XY; \quad  \qquad B(X,Y)=\displaystyle{-\frac{1}{\lambda }g(A_{\overline N}X,Y)}; \\
{\widetilde A}_{\overline N}X=A_{\overline N}X; \quad \qquad \tau (X)=\displaystyle{\frac{1}{\lambda }(X\circ \lambda )}.
\end{array}
\end{equation}
The equality (\ref{4.24}) implies $A_{\overline N}X\in {\D}$ and having in mind (\ref{52}) we have
\begin{equation}\label{4.28}
\varphi (A_{\overline N}X)=\overline J(A_{\overline N}X), \qquad \forall \, X\in \Gamma (TM).
\end{equation}
From (\ref{4.2}), (\ref{4.10}) by using (\ref{4.25}), (\ref{4.27}) and (\ref{4.28}) we obtain
\begin{equation}\label{4.29}
A^*_{\overline \xi }X=-\varphi (A_{\overline N}X); \quad \qquad C(X,PY)=\lambda g(\varphi (A_{\overline N}X),Y).
\end{equation}
We close this section by some geometrical characterizations of both $(M,\widetilde g)$ and
$(M,\varphi ,\overline \xi ,\overline \eta ,g)$. It is well known \cite{MMDis}, if an almost contact manifold with $B$-metric belongs to the class ${\F}_1\oplus \ldots \oplus {\F}_4\oplus {\F}_5\oplus {\F}_6\oplus {\F}_9\oplus {\F}_{10}\oplus {\F}_{11}$, then the contact distribution ${\D}$ of the manifold is an integrable distribution. Hence, if we suppose that the 
hypersurface $(M,\varphi ,\overline \xi ,\overline \eta ,g)$ of $\overline M$ belongs to the class 
${\F}_4\oplus {\F}_5\oplus {\F}_6$, then its contact distribution ${\D}$ is integrable As ${\D}$ is the screen distribution
$S(TM)$ of the corresponding hypersurface $(M,\widetilde g)$ to $(M,\varphi ,\overline \xi ,\overline \eta ,g)$, we have that
$S(TM)$ is integrable, too. We will show that the converse statement is also true. Let 
$(M,\varphi ,\overline \xi ,\overline \eta ,g)$ be a hypersurface of $\overline M$ such that its contact distribution ${\D}$
is integrable, which is equivalent to the assumption the screen distribution $S(TM)$ of $(M,\widetilde g)$ is integrable.
Then \corref{Corollary 4.1} implies $\widetilde A_V(\overline J(PX))=\overline J(\widetilde A_VPX)$ for any
$X\in \Gamma (TM)$. From the last equality, by using (\ref{52}), (\ref{4.24}), (\ref{4.28}) and $\widetilde A_V=A_V$ on $TM$, we
obtain
\begin{equation}\label{4.30}
A_V(\varphi X)=\varphi (A_VX), \qquad \forall \, X\in \Gamma (TM).
\end{equation}
Taking into account (\ref{4.24}), (\ref{4.30}) and (\ref{4.21}) we conclude that $(M,\varphi ,\overline \xi ,\overline \eta ,g)$
belongs to the class ${\F}_4\oplus {\F}_5\oplus {\F}_6$. So, we establish the truth of the following
\begin{prop}\label{5.1}
The assertion $(M,\varphi ,\overline \xi ,\overline \eta ,g)$ belongs to the class ${\F}_4\oplus {\F}_5\oplus {\F}_6$ is
equivalent to each of the following assertions: 
\begin{enumerate}
\renewcommand{\labelenumi}{(\roman{enumi})}
\item The contact distribution ${\D}$ of $(M,\varphi ,\overline \xi ,\overline \eta ,g)$ is integrable.
\item The screen distribution $S(TM)$ of $(M,\widetilde g)$ is integrable.
\end{enumerate}
\end{prop}
Now, we prove
\begin{prop}\label{5.2}
The following assertions are equivalent:
\begin{enumerate}
\renewcommand{\labelenumi}{(\roman{enumi})}
\item $(M,\varphi ,\overline \xi ,\overline \eta ,g)$ is totally geodesic.
\item $(M,\widetilde g)$ is totally geodesic.
\item $(M,\varphi ,\overline \xi ,\overline \eta ,g)$ belongs to the class ${\F}_0$.
\end{enumerate}
\end{prop}
\begin{proof}
As $(M,\varphi ,\overline \xi ,\overline \eta ,g)$ is a non-degenerate hypersurface of $\overline M$, it is totally geodesic if
and only if the shape operator $A$ vanishes identically. The hypersurface $(M,\widetilde g)$ is lightlike and it is totally
geodesic (\cite{D-B}, \cite{D-S}) if and only if $B(X,Y)=0$ for any $X,Y\in \Gamma (TM)$, which according to (\ref{4.27}) is equivalent to $g(A_{\overline N}X,Y)=0$ on $TM$. Having in mind that $g$ is non-degenerate on $M$, the last equality is 
fulfilled if and only if $A=0$. Finally, from (\ref{4.21}) we have that the characterization condition of the class ${\F}_0$ is
$A=0$. Thus, each of the assertions ({\rm i}), ({\rm ii}) and ({\rm iii}) is equivalent to the condition $A=0$, which completes the proof.
\end{proof}
\begin{prop}\label{5.3}
$(M,\widetilde g)$ is totally umbilical iff $(M,\varphi ,\overline \xi ,\overline \eta ,g)$ belongs to the class
${\F}_5$.
\end{prop}
\begin{proof}
Let $(M,\widetilde g)$ be totally umbilical. By using (\ref{4.16}), (\ref{4.25}), (\ref{4.29}) and (\ref{4.24}) we obtain
$A_{\overline N}X=-\lambda \rho \varphi X$ for any $X\in \Gamma (TM)$. From the last equality we find $\lambda \rho=
\displaystyle{\frac{{\rm tr}(A_{\overline N}\circ \varphi )}{2n-2}}$ and hence
\begin{equation}\label{4.31}
A_{\overline N}X=-\displaystyle{\frac{{\rm tr}(A_{\overline N}\circ \varphi )}{2n-2}}\varphi X.
\end{equation}
According to (\ref{4.21}) from (\ref{4.31}) it follows $(M,\varphi ,\overline \xi ,\overline \eta ,g)\in {\F}_5$. Conversely, if
$(M,\varphi ,\overline \xi ,\overline \eta ,g)\in {\F}_5$, then the equality (\ref{4.31}) is valid. By using (\ref{4.29}) and
(\ref{4.25}) we obtain $A^*_\xi PX=\displaystyle{\frac{\theta ^*(\overline \xi )}{\lambda (2n-2)}PX}$, which implies 
$(M,\widetilde g)$ is totally umbilical.
\end{proof}
Analogously we establish the truth of the following
\begin{prop}\label{5.4}
The screen distribution $S(TM)$ of $(M,\widetilde g)$ is totally umbilical iff $(M,\varphi ,\overline \xi ,\overline \eta ,g)$ belongs to the class ${\F}_4$.
\end{prop}

\section{Examples of radical transversal lightlike hypersurfaces of almost complex manifolds with Norden metric}\label{sec-6}
\begin{example}\label{6.1}
We consider $\mathbb{R}^{2n+2}=\{\left(u^1,\ldots ,u^{n+1};v^1,\ldots ,v^{n+1}\right)\vert u^i, v^i\in \mathbb{R}\}$ as a complex
Riemannian manifold with the canonical complex structure $\overline J$. In \cite{GaMGri} a metric $\overline g$ on 
$(\mathbb{R}^{2n+2},\overline J)$ was defined by
\[
\overline g(X,X)=-\delta _{ij}\lambda ^i\lambda ^j+\delta _{ij}\mu ^i\mu ^j,
\]
where $\displaystyle{X=\lambda ^i\displaystyle{\frac{\partial}{\partial u^i}+\mu ^i\frac{\partial}{\partial v^i}}}$. It is
easy to check that the canonical complex structure $\overline J$ is an anti-isometry with respect to $\overline g$ and hence 
$(\mathbb{R}^{2n+2},\overline J,\overline g)$ is a complex manifold with Norden metric. As usually, the associated metric to 
$\overline g$ is denoted by $\overline {\widetilde g}$. Identifying the point $p=\left(u^1,\ldots ,u^{n+1};v^1,\ldots ,
v^{n+1}\right)$ in $\mathbb{R}^{2n+2}$ with its position vector $Z$, in \cite{GaMGri} the following real hypersurface $M$ of
$\mathbb{R}^{2n+2}$ was defined 
\[
M: \quad \overline g(Z,\overline JZ)=0; \quad \overline g(Z,Z)={\rm ch}^2t, \quad t>0.
\]
It is clear that $\overline JZ$ is orthogonal to $TM$ with respect to $\overline g$, i. e. $\overline JZ \in TM^\bot $. 
For the vector field  $\overline N=(1/{\rm ch}t) \overline JZ$ we have 
$\overline g(\overline N,\overline N)=-1$, i. e. $\overline N$ is a time-like unit normal to $(M,g)$. Since 
$\overline g(\overline N,\overline J\overline N)=0$, the vector field $\overline J\overline N$ is a space-like unit, which
belongs to $TM$. Hence, $(M,g)$ is a non-degenerate hypersurface of $\mathbb{R}^{2n+2}$ satisfying the conditions (\ref{4.18}).
Then from \thmref{Theorem 5.1} it follows that $(M,\widetilde g)$ is a radical transversal lightlike hypersurface of 
$\mathbb{R}^{2n+2}$ such that $TM^{\widetilde \bot }=span\{\overline J\overline N\}$, ${\rm tr}(TM)=span\{\overline N\}$ and the
screen distribution $S(TM)$ coincides with the complementary orthogonal with respect to $\overline g$ vector subbundle of 
$\overline J(TM^\bot )$ in $TM$. Taking into account that $(\mathbb{R}^{2n+2},\overline J,\overline g)$ is a complex manifold
with Norden metric, from \thmref{Theorem 3.3} we have that $(M,\widetilde g)$ is a CR-manifold.
\par
Moreover, in \cite{GaMGri} were defined an almost contact structure $(\varphi ,\overline \xi ,\overline \eta )$ and B-metric
$g$ on $M$ by (\ref{4.20}) and it was proved that $(M,\varphi ,\overline \xi ,\overline \eta ,g)$ is an almost contact
manifold with B-metric in the class ${\F}_5$. According to \propref{Proposition 5.4} it follows that $(M,\widetilde g)$ is
totally umbilical.
\end{example}

\begin{example}\label{6.2}
We consider the Lie group $GL(2;\mathbb{R})$ with a Lie algebra $gl(2;\mathbb{R})$. 
The real Lie algebra $gl(2;\mathbb{R})$ is spanned by
the left invariant vector fields \\
$\{X_1,X_2,X_3,X_4\}$, where we set
\[
\hspace*{-4mm}
X_1=\left(\begin{array}{ll}
1 & 0
\cr 0 & 0 
\end{array}\right) , \, \,
X_2=\left(\begin{array}{ll}
0 & 1
\cr 0 & 0 
\end{array}\right) , \, \,
X_3=\left(\begin{array}{ll}
0 & 0
\cr 1 & 0 
\end{array}\right) , \, \,
X_4=\left(\begin{array}{ll}
0 & 0
\cr 0 & 1 
\end{array}\right) .
\]
We define an almost complex structure $\overline J$ and a left invariant metric $\overline g$ on $gl(2;\mathbb{R})$ by
\begin{equation}
\overline JX_1=X_4, \quad \overline JX_2=X_3, \quad \overline JX_3=-X_2, \quad \overline JX_4=-X_1
\label{3.53'}
\end{equation}
and
\begin{equation}
\begin{array}{ll}
\overline g(X_i,X_i)=-\overline g(X_j,X_j)=-1; \quad i=1,3; \quad j=2,4;\\
\overline g(X_i,X_j)=0; \quad i\neq j; \quad i,j=1,2,3,4.
\end{array}
\label{3.53''}
\end{equation}
Using (\ref{3.53'}) and (\ref{3.53''}) we check that the metric $\overline g$ is a Norden metric and consequently
$(GL(2;\mathbb{R}),\overline J,\overline g,\overline {\widetilde g})$ is a 4-dimensional almost complex
manifold with Norden metric.
\par
The real special linear group $SL(2;\mathbb{R})=\{A\in GL(2;\mathbb{R}): det(A)=1\}$ 
is a Lie subgroup of $GL(2;\mathbb{R})$ with a Lie algebra $sl(2;\mathbb{R})$ of all $(2\times 2)$ real traceless matrices.  
The Lie algebra $sl(2;\mathbb{R})$ is a 3-dimensional subalgebra of  $gl(2;\mathbb{R})$, spanned by $\{X_1-X_4, X_2, X_3\}$. 
Thus $SL(2;\mathbb{R})$ is a hypersurface of $GL(2;\mathbb{R})$.
We find that the normal space (with respect to $\overline g$) $sl(2;\mathbb{R})^\bot $ is spanned by $\{X_1-X_4\}$. Hence $sl(2;\mathbb{R})\cap sl(2;\mathbb{R})^\bot =sl(2;\mathbb{R})^\bot =span\{\xi =X_1-X_4\}$, 
i.e. $(SL(2;\mathbb{R}),g)$ is a lightlike hypersurface of $GL(2;\mathbb{R})$. We choose a holomorphic with respect to 
$\overline J$ screen distribution $S(sl(2;\mathbb{R}))$, spanned by $\{X_2, X_3\}$. From \thmref{Theorem 3.1}
it follows that $(SL(2;\mathbb{R}),g)$ is a radical transversal lightlike hypersurface and ${\rm tr}(sl(2;\mathbb{R}))$ is
spanned by $\left\{N=\displaystyle{\frac{-X_1-X_4}{2}}\right\}$.

\end{example}


\end{document}